\documentclass[11pt,a4paper]{amsart}
\usepackage{amsmath,amssymb,amscd,amsthm,amsxtra,array,soul}

\usepackage[latin9]{inputenc}
\usepackage{graphicx}

\usepackage[usenames,dvipsnames]{xcolor}
\usepackage[
colorlinks=true,
linkcolor=blue,
citecolor=blue,
urlcolor=blue,
]{hyperref}

\newtheorem{theorem}{Theorem}[section]
\newtheorem{lemma}[theorem]{Lemma}

\theoremstyle{definition}
\newtheorem{definition}[theorem]{Definition}

\theoremstyle{remark}
\newtheorem{remark}[theorem]{Remark}
\numberwithin{equation}{section}

\DeclareMathOperator{\BMO}{BMO}
\DeclareMathOperator{\bmo}{bmo}
\DeclareMathOperator*{\supp}{supp}

\def\R{{\mathbb R}}
\newcommand{\RN }{\mathbb R^n}

\def\Xint#1{\mathchoice
  {\XXint\displaystyle\textstyle{#1}}%
  {\XXint\textstyle\scriptstyle{#1}}%
  {\XXint\scriptstyle\scriptscriptstyle{#1}}%
  {\XXint\scriptscriptstyle\scriptscriptstyle{#1}}%
  \!\int}
\def\XXint#1#2#3{{\setbox0=\hbox{$#1{#2#3}{\int}$}
    \vcenter{\hbox{$#2#3$}}\kern-.5\wd0}}

\def\avgint{\Xint-}
\newcommand{\vertiii}[1]{{\left\vert\kern-0.25ex\left\vert\kern-0.25ex\left\vert #1 
    \right\vert\kern-0.25ex\right\vert\kern-0.25ex\right\vert}}

\numberwithin{equation}{section}

\allowdisplaybreaks

\begin{document}

\title[Minimal BMO]{Minimal conditions for BMO}

\author{Javier Canto}
\address[Javier Canto]{ BCAM \textendash  Basque Center for Applied Mathematics, Bilbao, Spain}
\email{jcanto@bcamath.org}

\author{Carlos P\'erez}
\address[Carlos P\'erez]{ Department of Mathematics, University of the Basque Country, IKERBASQUE 
(Basque Foundation for Science) and
BCAM \textendash  Basque Center for Applied Mathematics, Bilbao, Spain}
\email{cperez@bcamath.org}

\author{Ezequiel Rela}
\address[Ezequiel Rela]{Department of Mathematics,
Facultad de Ciencias Exactas y Naturales,
University of Buenos Aires, Ciudad Universitaria
Pabell\'on I, Buenos Aires 1428 Capital Federal Argentina} \email{erela@dm.uba.ar}

\thanks{ J.C. and C.P. are supported by the Ministerio de Econom\'ia y Competitividad (Spain) through grant MTM2017-82160-C2-1-P and SEV-2017-0718, and by Basque Government through grant IT-641-13 and BERC 2018-2021. }

\thanks{J.C. is also supported by Basque Government through ``Ayuda para la formaci\'on de personal investigador no doctor".}

\thanks{C.P. is supported by IKERBASQUE.}

\thanks{E.R. is partially supported by grants UBACyT 20020170200057BA, PIP (CONICET) 11220110101018, by the Basque Government through the BERC 2014-2017 program, and by the Spanish Ministry of Economy and Competitiveness MINECO: BCAM Severo Ochoa accreditation SEV-2013-0323. This project has received funding from the European Union's Horizon 2020 research and innovation programme under the Marie Sklodowska-Curie grant agreement No 777822}



\begin{abstract}
We study minimal integrability conditions via Luxemburg-type expressions  with respect to generalized oscillations that imply the membership of a given function $f$ to the space $\BMO$. Our method is simple, sharp and flexible enough to be adapted to several different settings, like spaces of homogeneous type, non doubling measures on $\R^n$ and also $\BMO$ spaces defined over more general bases than the basis of cubes.
\end{abstract}

\maketitle

\section{Introduction and Main Results}\label{sec:intro}

One of the classical results concerning the space of Bounded Mean Oscillation ($\BMO$) is the John--Nirenberg theorem, originally proved in \cite{JN}. There, a self-improvement property was established for functions in $\BMO$, providing a local exponential integrability estimate. Moreover, no better self-improvement can be found, so the John--Nirenberg theorem is the maximal integrability condition for $\BMO$. 

The main concern of this article is precisely the opposite problem: instead of studying self-improvement properties with $\BMO$ as an starting point, we want to find how much we can weaken the initial starting point but still self-improve back to $\BMO$. More precisely, we show that the membership of a given function to $\BMO$ can be obtained from a much weaker condition on generalized averages defined by  Luxemburg type norms.

Even though this problem was already addressed in a qualitative fashion by John in \cite{John1965} and later by Str\"omberg in \cite{Stromberg79}, our point of view is more quantitative, motivated by the recent work \cite{LSSVZ-BMO} which in turn was motivated by \cite{LongYang}. Our results extend those in \cite{LSSVZ-BMO}, giving more precise estimates that can also be applied to different contexts such as spaces of homogeneous type or non-doubling measures in $\R^n$.

Let us start by recalling that a function $f$ belongs to $\BMO$ if 
\begin{equation*}
\|f\|_{\BMO}:=\sup_{Q}\avgint_Q |f-f_Q|<\infty,
\end{equation*}
where the supremum is taken over all cubes with sides parallel to the coordinate axes and $f_Q$ stands for the average of the function $f$ with respect to the cube $Q$. Note that the definition itself asks for the local integrability of the function $f$ just to test the condition. We also recall that 
\begin{equation}\label{eq:BMOclassic-vs-BMOinf}
\|f\|_{\BMO}\: \simeq \: \sup_{Q}\inf_{c\in \R}\avgint_Q |f-c|,
\end{equation}
so, in principle, one could try to test the above condition without assuming a priori local integrability.

We will consider similar averages  via Luxemburg expressions with respect to functions $\varphi:[0,\infty)\to[0,\infty)$ such that $\varphi(0)=0$.

Given a cube $Q$ and a function $\varphi$, we introduce the following notation:
\begin{equation}\label{eq:luxemburg}
 \|f\|_{\varphi,Q} :=
\inf\left\{  \lambda > 0 : \frac{1}{|Q|}
\int_Q \varphi\left(\frac{|f|}{\lambda}\right)\,dx \leq 1 \right\}.
\end{equation}
Note that this quantity is homogeneous  $\|\lambda f\|_{\varphi,Q}=\lambda \|f\|_{\varphi,Q}$, $\lambda>0$,  but it is not a norm in general except when $\varphi$ is convex.

Using the Luxemburg expression \eqref{eq:luxemburg}, we define the  class $\BMO_\varphi$ as the set of measurable functions $f$ satisfying 
\begin{equation}\label{eq:BMO_varphi}
\|f\|_{\BMO_\varphi} := \sup_Q\inf_{c\in \mathbb{R}}  \|f-c\|_{\varphi,Q}<\infty.
\end{equation}
Clearly, we have that for $\varphi(t)=t,$ we have $\BMO_\varphi=\BMO$.
Observe that by definition  $\|f\|_{\BMO_\varphi} \leq K$ if and only if 
\begin{equation}\label{propertyhomog}
\inf_{c\in \mathbb{R}}  \frac{1}{|Q|} \int_Q \varphi\left(\frac{|f-c|}{K}\right)\,dx \leq 1
\end{equation}
for every cube $Q$.

We will focus on the special class of increasing and concave functions $\varphi$ in $[0,\infty)$ with $\varphi(0)=0$ and $\varphi (t) \rightarrow \infty$ as $t\rightarrow \infty$. Such functions must be continuous and subadditive, that is,
\[
\varphi(t_1+t_2) \leq \varphi(t_1) + \varphi(t_2).
\]

The first main result in this paper is the following.

\begin{theorem}\label{thm:Main-Cubes}
Let $\varphi$ be an increasing, concave function with $\varphi(0)=0$ and such that $\lim_{t\to \infty}\varphi(t)=+\infty$. Then $\BMO_\varphi=\BMO$ with the following quantitative estimates:
\begin{equation*}\label{eq:equiv-norm-Luxemburg}
\varphi^{-1}(1) \: \|f\|_{\BMO_\varphi} \leq \|f\|_{\BMO} \leq \big( 2\varphi^{-1} \big(4) + \varphi^{-1} \big( 2+2^{n+2}\big) \big) \: \|f\|_{\BMO_\varphi}.
\end{equation*}

\end{theorem}

\begin{remark} 
Although concavity of $\varphi$ is needed for the first inequality above, subadditivity is sufficient for the second inequality. This  observation could be useful for other circumstances or functions $\varphi$.
\end{remark}

Theorem \ref{thm:Main-Cubes} can be seen as an improvement of the main result from \cite{LSSVZ-BMO}. There, the authors deal with a quantity similar to \eqref{eq:BMO_varphi} defined as 
\begin{equation}\label{eq:LSSVZ-Kh}
K_{\varphi, Q}(f)=\sup _{J \text { subcube } Q}
\avgint_J \varphi\left(\left|f-f_J\right|\right).
\end{equation}
They obtain, under some conditions on $\varphi',\varphi''$ and $\varphi'''$, that the finiteness of $K_{\varphi, Q}(f)$ implies the membership of $f$ to $\BMO(Q)$. Their approach is based on the Bellman function method,  and they obtain quantitative upper and lower bounds on $\|f\|_{\BMO}$ in terms of \eqref{eq:LSSVZ-Kh}. However, their estimates are not homogeneous which might be a drawback for some applications.

Our proof here is based in the classical (dyadic) Calder\'on--Zygmund (CZ) decomposition at a local level on a given cube $Q$. The method is transparent and allows to precisely track the involved constants to give the result in Theorem \ref{thm:Main-Cubes} without any regularity hypothesis on $\varphi$. Furthermore, our proof yields homogeneous estimates and it does not require a priori local integrability for $f$.

We can go even further in the search for minimal conditions on the function $\varphi$. We mention that in \cite{LSSVZ-BMO}, the main result can be extended to almost any measurable function $\varphi$ going to infinity at infinity. Our method is also able to produce a similar result. 

\begin{theorem}\label{thm:Main-Cubes-general-phi}
Let $\psi:[0,\infty)\to[0,\infty)$ be any measurable function such that $\psi(0)=0$ and $\lim_{t\to \infty}\psi(t)=+\infty$. Then 
\begin{equation}\label{estimateGeneral-phi}
 \|f\|_{\BMO} \leq c_{n,\psi} \|f\|_{\BMO_\psi}.
\end{equation}
\end{theorem}

The method we present is flexible enough to also  solve the same problem in various different settings. We will prove the same result in the context of spaces of homogeneous type (SHT) where the space $(\mathbb X,d,\mu)$ is endowed with a quasi metric and a doubling measure, see Section \ref{sec:homogen} for the precise definitions.

For a function $\varphi$ we define the $\|f\|_{\BMO_\varphi(\mathbb X)}$ and the  corresponding class as:
\begin{equation*}
\|f\|_{\BMO_\varphi(\mathbb X)} := \sup_B \ \inf_{c\in \R} \ \inf \left \{\lambda>0 : \ \avgint_B \varphi\big( \frac{|f(x)-c|}{\lambda}\big) d\mu(x) \leq 1 \right \},
\end{equation*}
where the supremum is taken over all balls $B\subset \mathbb X$. We also define $\BMO(\mathbb X)$ with the quantity
\[
\|f\|_{\BMO(\mathbb X)} = \sup_B \ \inf_c \avgint_B |f-c|d\mu. 
\]

\begin{theorem}\label{thm:homogen}
Let $\varphi$ be as in Theorem \ref{thm:Main-Cubes}. Then $\BMO(\mathbb X)= \BMO_\varphi(\mathbb X)$ and 
\begin{equation*}
\varphi^{-1}(1) \: \|f\|_{\BMO_\varphi(\mathbb X)} \leq \|f\|_{\BMO(\mathbb X)} \leq c_{\varphi, \mu} \: \|f\|_{\BMO_\varphi(\mathbb X)}.
\end{equation*}
\end{theorem}

The proof of this theorem requires an adapted version of the classical CZ decomposition theorem and some other covering lemmas that we will develop accordingly.

We will also study the problem in $\R^n$ endowed with a quite general non doubling measure $\mu$. The usual requirement is to ask for the measure to be non atomic. In that case, it is known that there is an orthogonal system of coordinates such that $\mu(\partial(Q))=0$ for any cube $Q$ with sides parallel to the axes from that coordinate system, which is assumed to be the canonical one (see \cite{MMNO}). We mention, as an example of such measures, that a very natural choice satisfying these conditions is the class of measures with \emph{polynomial growth}, meaning that there exists a constant $C>0$ and a positive number $\alpha$ such that 
\begin{equation}\label{eq:Measure-Growht}
\mu(B(x,r))\le C r^\alpha \qquad x\in \supp(\mu).
\end{equation}

The natural definitions of $\BMO$ and $\BMO_\varphi$ in this context are the following. We will say that $f\in \BMO(\mu)$ if 
\begin{equation*}
\|f\|_{\BMO(\mu)}:=\sup_{Q}\avgint_Q |f-f_Q|\,d\mu<\infty,
\end{equation*}
and $f\in \BMO_\varphi(\mu)$ if 
\begin{equation*}
\|f\|_{\BMO_\varphi(\mu)}:=\sup_{Q}
\inf_{c\in \R}\inf\left\{  \lambda > 0 : \frac{1}{\mu(Q)}
\int_Q \varphi\left(\frac{|f-c|}{\lambda}\right)\,d\mu \leq 1 \right\}<\infty.
\end{equation*}

\begin{theorem}\label{thm:non-doubling}
Let $\varphi$ be as in Theorem \ref{thm:Main-Cubes}. Then, for any non-atomic measure $\mu$, we have that  $\BMO(\mu)= \BMO_\varphi(\mu)$ and 

\begin{equation*}
\varphi^{-1}(1) \: \|f\|_{\BMO_\varphi(\mu)} \leq \|f\|_{\BMO(\mu)} \leq c_{\varphi,n} \: \|f\|_{\BMO_\varphi(\mu)}.
\end{equation*}
\end{theorem}

The proof of the above theorem relies on a variation of the standard CZ decomposition and Besicovitch's covering theorem that we borrow from \cite{OP-nondoubling}. The precise statement is in Lemma \ref{lem:BCZ}.

So far, we can see (and it will become clear in the actual proof) that the heart of the matter is to have the correct version of a CZ decomposition adapted to the problem that we need to solve, taking into account the geometric features of the space (like in the case of SHT) or the nondoubling nature of the measure (like in Theorem \ref{thm:non-doubling}).

For the basis of rectangles in $\R^n$, the appropriate decomposition lemma is a very clever argument proven by Korenovskyy, Lerner and Stokolos in \cite{KLS} known as a generalized version of Riesz's Rising sun lemma. Using that lemma we can provide a proof that extends, in some sense, Theorem \ref{thm:Main-Cubes} and Theorem \ref{thm:non-doubling} at the same time: we can prove the analogue result for basis of rectangles and with non doubling measures. 

To present the result, we need to define here the ``little" $\bmo(\mu)$ space in the same way of the usual $\BMO(\mu)$ but with rectangles instead of cubes.  We refer the reader to the recent article \cite{Hart-Torres-19} for several results on this space.

\begin{definition}
Let $\varphi$ be an increasing function with $\varphi(0)=0$  and let $\mu$ be a Radon  measure. We denote by $\bmo_\varphi(\mu)$, little $\BMO_\varphi(\mu)$ the class of functions $f$ satisfying
\begin{equation*}
\|f\|_{ \bmo_\varphi(\mu) } :=\sup_R \inf_c   
\|f-c\|_{\varphi,R, \mu }\\
%
< \infty,
\end{equation*}
where the supremum is taken over all rectangles with sides parallel to the coordinate axes  and  the local averages are defined as in \eqref{eq:luxemburg} but with respect to the measure $\mu$, that is,
\begin{equation*}
 \|f\|_{\varphi,R,\mu} :=
\inf\left\{  \lambda > 0 : \frac{1}{|R|}
\int_R\varphi\left(\frac{|f|}{\lambda}\right)\,d\mu \leq 1 \right\}.
\end{equation*}
\end{definition}

\begin{theorem}\label{thm:rectangles-non-doubling}
Let $\varphi$ be as in Theorem \ref{thm:Main-Cubes}. Then, for any non-atomic measure $\mu$, we have that  $\bmo(\mu)= \bmo_\varphi(\mu)$ and 
\begin{equation*}
\varphi^{-1}(1) \: \|f\|_{\bmo_\varphi(\mu)} \leq \|f\|_{\bmo(\mu)} \leq c_{\varphi,n} \: \|f\|_{\bmo_\varphi(\mu)}.
\end{equation*}
\end{theorem}

\section{\texorpdfstring{$\BMO$}{BMO} through Luxemburg}\label{sec:BMO-luxemburg} 

One of the main tools in this work concerns Orlicz-type spaces. We refer to \cite{Wilson-LNM} for a general discussion of the theory.  Although the general theory of Orlicz spaces deals with convex functions, these spaces can be defined for quite general functions. Our concern in this work is with functions $\varphi$ which are concave, increasing and satisfy $\varphi(0)=0$ and $\varphi(t) \rightarrow\infty$ as $t \rightarrow \infty$. 

For a cube $Q$, the Orlicz-type space $L_\varphi(Q,\frac{dx}{|Q|})$ with respect to a function $\varphi$ is defined as the set of functions $f$ for which there exists some $\lambda>0$ such that 
\[
\avgint_Q \varphi\left (\frac{|f|}{\lambda}\right ) <\infty.
\]
This expression is not homogeneous, so we introduce the quantity
\begin{equation}\label{eq:Orlicz-norm-loc}
\|f\|_{\varphi,Q} = \inf \big\lbrace \lambda>0: \avgint_Q \varphi\Big( \frac{|f(x)|}{\lambda}\Big) dx \leq 1\big\rbrace.
\end{equation}
Due to the lack of convexity of $\varphi$, \eqref{eq:Orlicz-norm-loc} will not satisfy the triangular inequality in general and thus it is not a norm. However, we will use sometimes the expression ``norm" even though \eqref{eq:Orlicz-norm-loc} is not a norm in the usual sense. Nevertheless, using the concavity we can prove
\begin{equation}\label{eq:inequ-norms}
\|f\|_{\varphi,Q} \leq \frac{1}{ \varphi^{-1}(1)} \: \|f\|_{L^1(Q,\frac{dx}{|Q|})}.
\end{equation}
Indeed, one just needs to choose $\lambda$ as the right hand side of \eqref{eq:inequ-norms} and test \eqref{eq:Orlicz-norm-loc}, applying Jensen's inequality.

Finally, we define the appropriate $\BMO$ space in this context. A way of doing so might be to substitute the $L^1$ norm in the oscillation by  means of \eqref{eq:Orlicz-norm-loc}, that is,
\[ \sup_Q \|f-f_Q\|_{\varphi,Q}.\]
Here, we choose the alternate expression
\begin{eqnarray}\label{eq:BMO-fi-norm}
\|f\|_{\BMO_\varphi} & = &  \sup_Q \inf_{c} \|f-c\|_{\varphi,Q} \\
& = & \sup_Q \inf_c \inf\Big\lbrace \lambda >0: \avgint_Q \varphi\Big( \frac{|f(x)-c|}{\lambda}\Big) dx \leq 1 \Big \rbrace.\nonumber
\end{eqnarray}

The first infimum is taken over all constants $c$, which may be real or complex depending on the context.
We define $\BMO_\varphi$ as the class of functions such that the expression in \eqref{eq:BMO-fi-norm} is finite.

One easy but key observation is that, if $\|f\|_{\BMO_\varphi}\leq 1$, then for each $Q$ there exits a constant $c_Q$ such that
\begin{equation*}
\avgint_Q \varphi(|f-c_Q|) \leq 2.
\end{equation*}
This definition of $\BMO_\varphi$ can naturally be generalized to other contexts such as SHT, $\R^n$ with a more general measure or even the basis of rectangles.

\section{Proofs} \label{sec:proofs}

In this section we present the proofs of our results. We start by developing the main techniques in the simplest case: $\RN$ with the Lebesgue measure and the classical $\BMO$ space defined using cubes. Then, we will adapt this method to the other cases: SHT, non-doubling measures in $\RN$ and also rectangles. 

The method is very versatile, as it can be used also for different purposes. In \cite{PR-Poincare}, similar techniques were used to obtain a variety of Poincar\'e-type inequalities in several different settings and in \cite{CP-extension}, some extensions of the John--Nirenberg theorem were obtained also using similar techniques.

\subsection{The proof for the main result}\label{sec:concave-proof}

\begin{proof}[Proof of Theorem \ref{thm:Main-Cubes}]
The first inequality follows from the discussion in Section \ref{sec:BMO-luxemburg}, so we need only to prove the second one.

Let us fix a function $f\in \BMO_\varphi$ with norm one, and let us fix a cube $Q$. Then we can find a constant $c_Q$ such that 
\begin{equation}\label{eq:avg<2}
\avgint_Q \varphi(|f-c_Q|) \leq 2.
\end{equation}
Recall that the goal here is to bound the oscillation of $f$ uniformly over all cubes. To that end, we introduce the quantity
\begin{equation}\label{eq:X}
X = \sup_{Q \text{ cube} } \avgint_Q |f(x)-c_Q|\ dx,
\end{equation}
where $c_Q$ is such that \eqref{eq:avg<2} holds for $Q$.
Note that by the observation in \eqref{eq:BMOclassic-vs-BMOinf} it is enough to show that the bound claimed in the theorem holds for this quantity. At certain point we will need to manipulate this $X$, so we need to start by assuming  that it is finite. In order to do that, we will work with the following truncated quantity, that is,
\begin{equation}\label{eq:X_m}
X_m = \sup_{Q \text{ cube} } \avgint_Q \min\{|f(x)-c_Q|,m\}\ dx, \quad m\geq 1.
\end{equation}

We consider here the usual dyadic Calder\'on--Zygmund decomposition of $\varphi(|f-c_Q|)$ adapted to $Q$ at height $L>2$. The result is the collection $\{Q_j\}$ of maximal dyadic subcubes of $Q$ satisfying 
\begin{itemize}
\item $\displaystyle L< \avgint_{Q_j} \varphi(|f-c_Q|) \leq 2^n L,$

\

\item $\varphi(|f(x)-c_Q|) \leq L $ \quad a.e. $\displaystyle x\in Q \setminus \bigcup_j Q_j$,

\

\item $\displaystyle\frac{1}{|Q|}\sum_j |Q_j| \leq \frac{2}{L}.$
\end{itemize}
Now, let us fix a cube $Q_j$. For a point $x\in Q_j$, we have
\[
|f(x)-c_Q|\leq |f(x)-c_{Q_j}| + |c_Q-c_{Q_j}|,
\]
where $c_{Q_j}$ is a constant so that $\avgint_{Q_j}\varphi(|f-c_{Q_j}|)\leq 2.$ We bound the second term as follows:
\begin{eqnarray*}
|c_Q-c_{Q_j}| & = &\varphi^{-1} \Big( \avgint_{Q_j} \varphi(|c_Q-c_{Q_j}|)\Big) \\
& \leq &\varphi^{-1} \Big( \avgint_{Q_j} \varphi(|f(x) -c_Q|) dx + \avgint_{Q_j}\varphi(|f(x)-c_{Q_j}|)dx\Big) \\
& \leq & \varphi^{-1} \big( 2^n L +2 \big).
\end{eqnarray*}
Here we have used the definition of the norm $\|f\|_{\BMO_\varphi}$, the properties of the Calder\'on--Zygmund decomposition, and the fact that $\varphi$ is subadditive and $\varphi^{-1}$ increasing.

We now proceed to estimate 
$$\avgint_Q \min \big\{ |f -c_{Q}|,m \big\} dx,$$
for $m\in\mathbb{N}$. 
We split the cube into the two sets: $\bigcup_{j} Q_j$ and $Q\setminus \bigcup_{j} Q_j$. On the first one, we have a good pointwise estimate on the size of $f-c_Q$. On the second, we will use that the CZ cubes are disjoint and the previous estimate. We will use a basic but key inequality: for any choice of positive parameters $a,b$ and $m$, we have that $\min \big\{ a+ b, m \big\}\leq \min\big\{ a,m \big\}+b$. Now, we start by controlling the integral over $Q\setminus \bigcup Q_j$ as
\begin{equation*}
\frac{1}{|Q|} \int_{Q\setminus \cup Q_j}\min \big\{ |f -c_{Q}|,m \big\}\le \varphi^{-1}(L).
\end{equation*}
Taking this into account, we proceed to estimate the average over the cube as follows
\begin{align*}
\avgint_Q \min \big\{ |f -c_{Q}|,m \big\} &\leq \varphi^{-1} (L)+ \frac{1}{|Q|} \sum_j \int_{Q_j}\min \big\{ |f -c_{Q}|,m \big\}  \\
& = \varphi^{-1} (L)+ \frac{1}{|Q|} \sum_j |Q_j| \avgint_{Q_j} \min \big\{ |f -c_{Q}|,m \big\} .
\end{align*}
The average over $Q_j$ is controlled by using the key property about the minimum, namely
\begin{align*}
\avgint_{Q_j} \min \big\{ |f -c_{Q}|,m \big\} & \le \avgint_{Q_j} \min \big\{ |f -c_{Q_j}|+|c_Q-c_{Q_j}|,m \big\}\\
& \le \avgint_{Q_j} \min \big\{ |f -c_{Q_j}|,m \big\} + |c_Q-c_{Q_j}|\\
& \le  X_m + \varphi^{-1} \big( 2+2^n L\big).
\end{align*}
Therefore, collecting estimates we get 
\begin{align*}
\avgint_Q \min \big\{ |f -c_{Q}|,m \big\} &\leq\varphi^{-1} (L)+ \frac{1}{|Q|} \sum_j |Q_j|  \left ( X_m + \varphi^{-1} \big( 2+2^n L\big)\right )\\
& \leq  \varphi^{-1} \big(L)  + \frac{2X_m}{L} + \frac2{L} \varphi^{-1} \big( 2+2^n L\big),
\end{align*}
where $X_m$ is the quantity defined by  \eqref{eq:X_m}, which is trivially bounded by $m$. Then, we can also take the supremum on the left hand side to obtain 
\[ X_m \leq  \varphi^{-1} \big(L)  + \frac2{L} \varphi^{-1} \big( 2+2^n L\big) + \frac{2X_m}{L}.\]
Now take $L=4$ and absorb $X_m$ into the LHS,
\[X_m \leq  2\varphi^{-1} \big(4)  + \varphi^{-1} \big( 2+2^{n+2}\big),\]
and hence for any cube $Q$ and for any $m\in\mathbb{N}$,
$$
\avgint_Q \min \big\{ |f -c_{Q}|,m \big\}\leq  2\varphi^{-1} \big(4)  + \varphi^{-1} \big( 2+2^{n+2}\big),
$$
and letting $m\to \infty$ concludes the proof of the theorem.
\end{proof}

Now we present the proof of Theorem \ref{thm:Main-Cubes-general-phi}.  The main idea is to replace a general function $\psi$ going to $+\infty$ with a related function $\varphi$ for which we can apply our main theorem.

\begin{proof}[Proof of Theorem \ref{thm:Main-Cubes-general-phi}]

Let $\psi:[0,+\infty)\to [0,+\infty)$ be a function such that $\psi(0)=0$ and $\lim_{t\to\infty}\psi(t)=+\infty$. Just by using the hypothesis on the behavior of $\psi$ at infinity, we can find some non negative $t_0\in [0,\infty)$ (depending on $\psi$) and a \textbf{polygonal} function $\varphi: [0,+\infty)\to [0,+\infty)$ which will be concave for large values of $t$ and smaller than $\psi$. 

More precisely, we will have that $\varphi(t)=0$ for all $t\le t_0$ (we need to wait until $\psi$ goes away from zero). Then, for $t\ge t_0$, $\varphi$ will be constructed as a polygonal consisting of consecutive segments with endpoints $(t_n,n)$, $(t_{n+1},n+1)$ with $n\in\mathbb{N}$ chosen in such a way that the resulting polygonal is continuous, concave and such that $\varphi(t)\le \psi(t)$ for all $t\in [0,\infty)$.
Using this auxiliary function and since we have immediately that
\begin{equation*}
\|f\|_{\BMO_\varphi} \le \|f\|_{\BMO_{\psi}},
\end{equation*}
we will prove \eqref{estimateGeneral-phi} for the new function $\varphi$ instead of $\psi$.
An inspection of the proof of Theorem \ref{thm:Main-Cubes} shows that the key step is to obtain 
\begin{equation*}
|c_Q-c_{Q_j}| \le \varphi^{-1} \big( 2^n L +2 \big),
\end{equation*}
where the subadditivity is used. Here, we proceed as follows using the layer cake formula. Write $A(x)=|c_Q-f(x)|$ and $B(x)=|f(x)-c_{Q_j}|$, so 
\begin{eqnarray*}
\int_{Q_j} \varphi(|c_Q-c_{Q_j}|) \ dx& \le &
\int_{Q_j} \varphi(|c_Q-f(x)|+|f(x)-c_{Q_j}|)\ dx\\
& = & \int_0^\infty \varphi'(t)|\{x\in Q_j: A(x)+B(x)>t\}|\ dt=I\\
\end{eqnarray*}
Note that $\varphi$ is differentiable almost everywhere since it is a polygonal. We can split the integral at $t=2t_0$ to obtain
\begin{eqnarray*}
I& = &
\int_{0}^{2t_0} \varphi'(t)|\{ A+B>t\}|\ dt+\int_{2t_0}^\infty \varphi'(t)|\{ A+B>t\}|\ dt\\
& \le &
|Q_j|\varphi(2t_0)+\int_{2t_0}^\infty \varphi'(t)|\{ A>t/2\}|\ dt+
\int_{2t_0}^\infty \varphi'(t)|\{ B>t/2\}|\ dt\\
& = &
|Q_j|\varphi(2t_0)+2\int_{t_0}^\infty \varphi'(2u)|\{ A>u\}|\ du+
2\int_{t_0}^\infty \varphi'(2u)|\{ B>u\}|\ du.
\end{eqnarray*}
Now we use that the derivative function $\varphi'$ is non negative and decreasing in $(t_0,\infty)$,  and so we obtain
\begin{eqnarray*}
I & \le & 
|Q_j|\varphi(2t_0)+2\int_{t_0}^\infty \varphi'(u)|\{ A>u\}|\ du+
2\int_{t_0}^\infty \varphi'(u)|\{ B>u\}|\ du\\
& \leq &    |Q_j|\varphi(2t_0)+ 2\int_{0}^\infty \varphi'(u)|\{ A>u\}|\ du +  2\int_{0}^\infty \varphi'(u)|\{ B>u\}|\ du\\
& = &  |Q_j|\varphi(2t_0)+ 2\int_{Q_j} \varphi(|f(x) -c_Q|) dx + 2\int_{Q_j}\varphi(|f(x)-c_{Q_j}|)dx.
\end{eqnarray*}

Finally, dividing by the measure of $Q_j$ we obtain a similar estimate as in the original proof. Indeed, whenever $|c_Q-c_{Q_j}| \ge t_0$, we obtain
\begin{eqnarray*}
|c_Q-c_{Q_j}| & = &\varphi^{-1} \Big( \avgint_{Q_j} \varphi(|c_Q-c_{Q_j}|)\Big) \\
& \leq &\varphi^{-1} \Big(\varphi(2t_0)+  2\avgint_{Q_j} \varphi(|f -c_Q|) dx + 2\avgint_{Q_j}\varphi(|f-c_{Q_j}|)dx\Big) \\
& \leq & \varphi^{-1} \big(\varphi(2t_0)+  2^{n+1} L +4 \big),
\end{eqnarray*}
where $\varphi^{-1}$ is the inverse of $\varphi$ restricted to $[t_0,+\infty)$. Otherwise, we simply bound $|c_Q-c_{Q_j}| \le t_0$ with the obvious consequences over the final estimate.
From here, the proof follows the same steps as in Theorem \ref{thm:Main-Cubes} to obtain 
\begin{equation*}
\|f\|_{\BMO}\le c_{n,\varphi}\|f\|_{\BMO_\varphi} \le   c_{n,\psi} \|f\|_{\BMO_{\psi}}. \qedhere
\end{equation*}

\end{proof}

\subsection{The proof for spaces of homogeneous type} \label{sec:homogen}

Now we can move on to the context of SHT and provide a proof for Theorem \ref{thm:homogen}. 
For the sake of completeness, we will give the basic definitions in spaces of homogeneous type, without lingering too much on the details.
A space of homogeneous type is a triple $(\mathbb X,d,\mu)$ where $\mathbb X$ is the set, $d$ is a quasi-metric and $\mu$ is a doubling measure. More precisely, $d$ satisfies all the hypothesis for a distance except for the triangular inequality, which is satisfied with a constant $\kappa \geq 1$:
\[
d(x,y) \leq \kappa \: \big( d(x,z)+d(z,y) \big), \quad x,y,z \in \mathbb X.
\]
Moreover, by \cite{MS} we may assume that the (open) balls with respect to $\mu$ are measurable and that $\mu$ is doubling, that is, there exists $c_\mu>0$ such that 
\[
\mu\big(B(x,2r)\big) \leq c_\mu \: \mu\big( B(x,r)\big),
\]
for all $x\in \mathbb X$ and all $r>0$. If $c$ is the smallest constant for which this holds, the number $D = \log_2 c$ is usually called the doubling order of
$\mu$. Then by iterating, we have
\begin{equation}\label{reiterat}
\frac{\mu(B)}{\mu(P)} \leq
c_{\mu,\kappa}\,\left(\frac{r(B)}{r(P)}\right)^D, 
\end{equation}
for every pair $P, B$ of balls such that $P \subset
B$.

Before the proof of Theorem \ref{thm:homogen}, let us state a few lemmas that will be used throughout the proof.

\begin{lemma}[Vitali covering in this context]\label{lem:Vitali} Let $\mathcal B$ be a collection of balls in $\mathbb X$ with bounded radius. There exists a subcolection $\mathcal B^* \subset \mathcal B$ of pairwise disjoint balls such that 
\[
\bigcup_{B\in \mathcal B} B \subseteq \bigcup_{B \in \mathcal B^*} \kappa (4\kappa +1) B.
\]
\end{lemma}
In view of Lemma \ref{lem:Vitali}, we define, for a ball $B$, the dilation
\[
B^* = \kappa ( 4\kappa +1) B.
\]
We also fix the following notation for dilations. Fix $\gamma>\kappa$ and we set
\[
\tilde B := \gamma \ B.
\]
This is needed because when doing the Vitali covering, dilating the balls may result in going outside the original ball $B$, but the following lemma guaranties that the dilated balls stay inside of $\tilde B$. For a ball $B$ we denote by $x_B$ and $r(B)$ the center and radius of $B$  respectively.

\begin{lemma}\label{lem:Bolas}
Let $B$ be a ball and let $\varepsilon>0$. There exists $L>1$ big enough so that if $P$ is another ball with center in $B$ and satisfying
\[
\mu(P) \leq \frac{\mu(\tilde B)}{L},
\]
then $r(P) \leq \varepsilon r(B)$. If $\varepsilon$ is small enough, this also implies $P^* \subset \tilde B$.
\end{lemma}

\begin{proof}
By contradiction, suppose that there exists some $\alpha>1$ such that $r(P) \geq \alpha r(B)$ with $\alpha$ independent from $L$. This implies $\tilde B \subset \kappa (\gamma+1)\frac 1\alpha P$. Indeed, for $y\in \tilde B$,
\begin{align*}
d(y,x_P) & \leq \kappa \big( d(y,x_B)+d(x_B,x_P)\big) \\
& \leq \kappa \big( \gamma r(B) + r(B) \big) \\%
& \leq \kappa (\gamma+1)  \frac 1 \alpha \ r(P).
\end{align*}
This bound on the radii will imply a bound on the measures. Indeed, by \eqref{reiterat},
\begin{eqnarray*}
\mu(\tilde B)  \leq \mu \big( \frac{  (\gamma+1)\kappa}{ \alpha} P\big) & \leq &c_{\mu,\kappa}\, \Big( \frac{  (\gamma+1)\kappa}{ \alpha} \Big)^{D} \mu(P)\\
&  \leq & \frac{c_{\mu,\kappa}}L  \Big( \frac{ (\gamma+1) \kappa }{\alpha} \Big)^{D} \mu(\tilde B ).
\end{eqnarray*}
This implies that $  c_{\mu,\kappa}\,\Big( \frac{ (\gamma+1) \kappa }{\alpha} \Big)^{D} \geq L$ which is not possible for $L$ big enough.

Now we prove the last statement. We set $y \in P^*$ and we want to see $y\in \gamma B=\tilde B$. Indeed,
\begin{align*}
d(y,c_B) & \leq \kappa \big( d(y,c_P) + d(c_P,c_B) \big) \\
& \leq \kappa \big( \kappa (4\kappa +1) \varepsilon r(B) + r(B) \big) \\
& \leq \kappa \big( \kappa (4\kappa+1) \varepsilon +1) r(B).
\end{align*}
Now, since $\gamma > \kappa,$ there exists $\varepsilon>0$ small enough such that 
\[\kappa \big( \kappa (4\kappa+1) \varepsilon +1)\leq \gamma.\]
Thus, $y\in \tilde B$ and we are done.
\end{proof}

\begin{proof}[Proof of Theorem \ref{thm:homogen}]
Assume that $\|f\|_{\BMO_\varphi(\mathbb X)}=1$. Set for a ball $P$ a constant $c_P$ such that 
\[ \avgint_P \varphi(|f-c_P|) \leq 2. \]
We are going to set $X$ in  a slightly different way from before, namely
\[
X = \sup_P \avgint_P  |f-c_{\tilde P}|.
\]
Notice that the ball of the integral and the one inside are related but not the same. Nevertheless, it is clear that $ \|f\|_{\BMO} \leq X$.
As in the proof of Theorem \ref{thm:Main-Cubes}, the hypothesis $X<\infty$ will be needed. This can be obtained via a truncation argument as in that proof, but we omit it for the sake of clarity. Thus, we may assume that $X<\infty$.

Now let us begin with the actual proof. Fix a ball $B$ and $L>1$ to be precised later. We make a decomposition in balls of the function $\varphi(|f-c_{\tilde B}|)$ in the spirit of Calder\'on--Zygmund and using the Vitali covering. By that, we mean the following process. 

We are going to make a covering by balls of the set 
\[\Omega_L = \{ x\in B: \varphi(|f(x)-c_{\tilde B}|)>L\} .\]
By the Lebesge differentiation theorem, for any $x\in \Omega_L$, there exists a ball $B_x$ centered at $x$ and contained in $\tilde B$ and such that
\begin{equation}\label{eq:condicion-L-homog}
\avgint_{B_x} \varphi(|f(y)-c_{\tilde B}|) d\mu(y) >L.
\end{equation}
Moreover, we can choose this $B_x$ to be maximal with respect to the radius. That is, any other ball $B_x'\subset \tilde B$ satisfying \eqref{eq:condicion-L-homog} must also satisfy $r(B_x')\leq 2r(B_x)$. This can be done since all balls contained in $\tilde B$ have bounded radius.

Now we have a family $\mathcal B = \{ B_x\}_x$ and we apply the Vitali Lemma \ref{lem:Vitali} to get a ``maximal" subfamily $\mathcal B' = \lbrace B_j\rbrace$. If $L$ is big enough, we can apply Lemma \ref{lem:Bolas} and this ensures that $B_j^* \subset \tilde B$ and, by the maximality of the radius of each of the $B_j$, since $r(B_j^*) \geq 2r(B_j),$ 
\[
\avgint_{B_j^*} \varphi(|f(y)-c_{\tilde B}|) d\mu(y) \leq L.
\]

Moreover, we have the estimate
\[
\sum_j \mu(B_j) \leq \frac 1L \sum_j \int_{B_j} \varphi(|f-c_{\tilde{B}}|) \leq \frac 1 L \int_{\tilde B} \varphi(|f-c_{\tilde B}|) \leq \frac 2{L} \mu(\tilde B) \leq \frac{C_{\mathbb X}}L \mu(B),
\]
where $C_\mathbb{X}$ denotes a constant depending on the doubling property of $\mu$.

Let us summarize all the properties of the family $\{B_j\}$:
\begin{itemize}
\item The balls $B_j$ are pairwise disjoint and all contained in $\tilde B$.
\item $\Omega_L \subset \cup_j B_j^*$
\item The balls $B_j^*$ are contained in $\tilde B$ and $\avgint_{B_j^*}\varphi(|f-c_{\tilde B}|)\leq L$.
\item $\sum_j \mu(B_j^*) \leq C_\mathbb{X} \sum_j \mu(B_j) \leq \frac{C_\mathbb{X}}{L} \mu(B).$
\end{itemize}

Now we begin to estimate $\avgint |f-c_{\tilde B}|$.
\begin{eqnarray*}
\avgint_B |f-c_{\tilde B}| & \leq &\frac{1}{\mu(B)} \int_{(\Omega_L)^c} |f-c_{\tilde B}| +  \frac{1}{\mu(B)} \int_{\Omega_L} |f-c_{\tilde B}| \\
& \leq &\varphi^{-1} (L) + \frac1{\mu(B)} \sum_j \int_{B_j^*} |f(y)-c_{\tilde B}| d\mu(y) \\
& \leq &\varphi^{-1} (L) +\frac1{\mu(B)} \sum_j  \int_{B_j^*}\Big(  |f(y)-c_{\widetilde{B_j^*}}| + |c_{\tilde B} - c_{\widetilde{B_j^*}}| \Big) d\mu(y) \\
&= &(*).
\end{eqnarray*}
We now estimate $|c_{\tilde B} - c_{\widetilde{B_j^*}}|$:
\begin{eqnarray*}
|c_{\tilde B} - c_{\widetilde{B_j^*}}| & = &\varphi^{-1} \Big( \avgint_{B_j^*} \varphi(|c_{\tilde B} - c_{\widetilde{B_j^*}}|)d\mu(y) \Big) \\
& \leq & \varphi^{-1} \Big( \avgint_{B_j^*} \varphi(|f(y) - c_{\tilde B}|)d\mu(y) + \avgint_{B_j^*} \varphi(|f(y) -c_{\widetilde{B_j^*}}|)d\mu(y) \Big) \\
& \leq & \varphi^{-1} \Big( L + C_{\mathbb X} \avgint_{\widetilde{B_j^*}} \varphi(|f(y) -c_{\widetilde{B_j^*}}|)d\mu(y) \Big) \\
& \leq  & \varphi^{-1} \Big( L +2 C_{\mathbb X}\Big).
\end{eqnarray*}
And therefore, 
\begin{eqnarray*}
(*)& \leq & \varphi^{-1} (L) + \frac1{\mu(B)} \sum_j \int_{B_j^*}\Big(  |f(y)-c_{\widetilde{B_j^*}}| + \varphi^{-1} \big( L +2 C_{\mathbb X}\big) \Big)  d\mu(y) \\
&  \leq & \varphi^{-1} (L) +  \sum_j \frac{\mu(B_j^*) }{\mu(B)}\avgint_{B_j^*}\Big(  |f(y)-c_{\widetilde{B_j^*}}| + \varphi^{-1} \big( L +2 C_{\mathbb X}\big) \Big) d\mu(y) \\
&  \leq &\varphi^{-1} (L) + C_\mathbb{X}\sum_j \frac{\mu(B_j) }{\mu(B)}\Big( \avgint_{B_j^*}  |f(y)-c_{\widetilde{B_j^*}}| d\mu(y) + \varphi^{-1} \big( L +2 C_{\mathbb X}\big) \Big) \\
& \leq & \varphi^{-1}(L) +C_{\mathbb X}  \frac{X}{L }+ \frac{C_{\mathbb X}}{L} \varphi^{-1}(L+2C_{\mathbb X})  \\
&\leq &C_{L,\mathbb{X},\varphi} + C_{\mathbb X}  \frac{X}{L }.
\end{eqnarray*}
In order to finish, we take the supremum on the left, choose $L$ big enough and argue as in the euclidean case.
\end{proof}

\subsection{The proof for nondoubling measures}\label{sec:non-doubl}

The objective of this section is to prove Theorem \ref{thm:non-doubling}. Let's consider a nondoubling measure satisfying the growth condition \eqref{eq:Measure-Growht}. Therefore, it is non atomic and by \cite[Theorem 2]{MMNO} we can choose a coordinate system such that $\mu(\partial Q)=0$ for every cube $Q$ defined over that system.

We will present the proof for $n=1$ separately, since the situation there is much easier than in higher dimensions.
The heart of our main argument is the CZ decomposition. Here, in the nondoubling setting, we will abandon the metric to split the cubes and use the measure instead. We will construct a $\mu$-dyadic grid of subintervals such that every interval $I$ is divided into two subintervals each one of half of the measure of $I$. We sketch here the construction.

The first generation $G_1(I)$ of the dyadic grid  consists of the two disjoint subintervals $I_+$, $I_-$ of $I$ satisfying $\mu(I_+) = \mu(I_-) = \mu(I)/2$ (note that this partition may be non unique, in that case we choose the one that maximizes the length of $I_-$, just to fix a criteria) The next generation is $G_2(I)$ is $G_1(I_+)\cup G_1(I_-)$ and then the construction procedure continues recursively. Recall that the measure is non atomic, so we can take closed intervals sharing the endpoints. We denote by $\mathcal{D}_{I}^\mu$ the family of all the dyadic intervals resulting from this procedure. 
A sequence of nested intervals in this grid will be called a \emph{chain}. That is, a chain $\mathcal{C}$ will be of the form $\mathcal{C}=\{J_i\}_{i\in\mathbb N}$ such that $J_i\in G_i(I)$, and $J_{i+1}\subset J_i$ for all $i\ge1$.  

We can define $\mathcal{C}_\infty:=\bigcap_{J\in \mathcal{C}} J$ as the \emph{limit set} of the chain $\mathcal{C}$. Then, we have that $\mathcal{C}_\infty$ 
could be either a single point or a closed interval of positive length. In any case, we clearly have that $\mu(\mathcal{C}_\infty)=0$. 
We need to get rid of those limit sets $\mathcal{C}_\infty$ of positive length, so we call them \emph{removable}. The argument here is that in the real line there are at most countable many of them and the whole union is also a $\mu$-null set. We denote by $\mathcal{R}$ the set of all chains with removable limits.
If we define
\begin{equation*}
 E:= I\setminus \bigcup_{\mathcal{C}\in \mathcal{R}}\mathcal{C_\infty},
\end{equation*}
we conclude that $\mu(I)=\mu(E)$ and, in addition, for any $x\in E$, there exists a chain of nested intervals shrinking to $x$. Therefore the grid $\mathcal{D}_{I}^\mu$ forms a differential basis on $E$. Also, the dyadic structure of the basis guarantees the Vitali covering property (see \cite[Ch.1]{guz75}) and therefore this basis differentiates $L^1(E)$. 

Associated to this grid we define a \emph{dyadic} maximal operator as follows. For any $x\in E$,
\begin{equation*}
M^{\mathcal{D}_{I}^\mu }f(x) 
=
\sup_{J\in \mathcal D_I^\mu}\avgint_J |f|\ d\mu,
\end{equation*} 
By a standard differentiation argument, we have that this maximal function satisfies that $f\le M^{\mathcal{D}_{I}^\mu } f$,   almost everywhere on $E$. 

Now we can proceed with the proof of the 1 dimensional case of Theorem \ref{thm:non-doubling}. Let us fix a function $f\in \BMO_\varphi(\mu)$ with norm one, and let us fix an interval  $I$. As before, we can find a constant $c_I$ such that 
\begin{equation*}
\avgint_I \varphi(|f-c_I|)\,d\mu \leq 2.
\end{equation*}
We define again the corresponding $X$ as
\begin{equation*}
X = \sup_{I \text{ interval} } \avgint_I |f(x)-c_I|\ d\mu.
\end{equation*}
As in the euclidean setting, a truncation argument allows us to assume $X<\infty$.
Using our $\mu$-dyadic construction, we can perform a Calder\'on--Zygmund decomposition of $\varphi(|f-c_I|)$ adapted to $I$ at height $L>2$. We then obtain a family $\{I_j\}$ of dyadic subintervals of $I$ satisfying 
\begin{itemize}
\item $\displaystyle L< \avgint_{I_j} \varphi(|f-c_I|)d\mu \leq 2 L,$

\

\item $\varphi(|f(x)-c_I|) \leq L $ \quad $\mu$-a.e. $\displaystyle x\in I \setminus \bigcup_j I_j$.

\

\item $\displaystyle\frac{1}{\mu(I)}\sum_j \mu(I_j) \leq \frac{2}{L}.$
\end{itemize}
Once we have this crucial decomposition, we can develop the same proof as in the case of the Lebesgue measure. On a fixed maximal interval $I_j$, we write again 
\[
|f(x)-c_I|\leq |f(x)-c_{I_j}| + |c_I-c_{I_j}|,
\]
where $c_{I_j}$ is a constant so that $\avgint_{I_j}\varphi(|f-c_{I_j}|)\leq 2.$ We obtain 
\begin{equation*}
|c_I-c_{I_j}|\le \varphi^{-1} \big(2 L +2 \big)
\end{equation*}
Following the same line of ideas, we can control the  averges to estimate the $\BMO{\mu}$ norm
\begin{eqnarray*}
\avgint_I |f-c_I| & \leq & \varphi^{-1} \big(L) + \frac{1}{\mu(I)} \sum_j \mu(I_j)\Big(  \avgint_{I_j} |f-c_{I_j}| + |c_I-c_{I_j}|\Big)\\
 & \leq & \varphi^{-1}\big(L) + \frac{1}{\mu(I)} \sum_j \mu(I_j)\Big(  X + \varphi^{-1} \big( 2L+2\big)\Big) \\
 & \leq & \varphi^{-1} \big(L)  + \frac{2}{L}X+\frac{2}{L}\varphi^{-1} \big( 2L+2\big)
\end{eqnarray*}
Finally, taking the supremum over all intervals on the left hand side and choosing $L=4$ we obtain 
\begin{equation*}
X \leq 2\varphi^{-1} \big(4)  + \varphi^{-1} \big(10\big),
\end{equation*}
which finishes the proof.

Now we present the proof for $n>1$. Again, the key step is to construct an adequate CZ decomposition with dyadic structure. The ideal tool can be found in the work from \cite{OP-nondoubling} and consists in the following combination of the CZ decomposition and Besicovitch's covering theorem. We include here the statement of that lemma.

\begin{lemma}[Besicovitch--Calder\'on--Zygmund decomposition]\label{lem:BCZ} Let $Q$ be a cube and let $g\in L^1_\mu(Q)$ be a nonnegative function. Also let $L$ be a positive number such that 
$\avgint_Q g\, d\mu<L$. Then there is a
family of quasidisjoint cubes $\{Q_j\}$ contained in Q satisfying
\begin{equation*}
\frac{1}{\mu\left(Q_{j}\right)} \int_{Q_{j}} g d \mu=L
\end{equation*}
for each $j$ and such that 
\begin{equation*}
\quad g(x) \leq L \quad \text{ for } x \in Q \setminus \bigcup_{j} Q_{j}, \quad \mu-a . e.
\end{equation*}
More precisely, we can write 
\begin{equation*}
\bigcup Q_j =\bigcup_{k=1}^{B(n)}\bigcup_{Q_j\in \mathcal{F}_k}Q_j,
\end{equation*}
where each $\mathcal F_k$ is a family of disjoint cubes selected from the original collection. The number $B(n)$ is a geometric constant depending only on the dimension $n$ known as the \emph{Besicovitch constant}.
\end{lemma}

We can now provide the proof for Theorem \ref{thm:non-doubling} in the remaining cases $n>1$. Let's start again with a function $f$ such that $\|f\|_{\BMO_\varphi(\mu)}=1$ and fix a cube $Q$ and the corresponding $c_Q\in \R$ giving us the initial estimate
 \begin{equation*}
\avgint_Q \varphi(|f-c_Q|)\,d\mu \leq 2.
\end{equation*}
We define again the corresponding $X$ as
\begin{equation*}
X = \sup_{Q \text{ cube } } \avgint_Q |f(x)-c_Q|\ d\mu.
\end{equation*}
Applying Lemma \ref{lem:BCZ} with $L>2$, we obtain a quite similar collection of cubes as in the previous case. Precisely, we obtain the family of cuasidisjoint cubes $\{Q_j\}$ satisfying
\begin{itemize}
\item $\displaystyle  \avgint_{Q_j} \varphi(|f-c_Q|) = L$,

\

\item $\varphi(|f(x)-c_Q|) \leq L $ \quad a.e. $\displaystyle x\in Q\setminus \bigcup_j Q_j$,

\

and a minor difference in the next property:
\item $\displaystyle\frac{1}{\mu(Q)}\sum_j \mu(Q_j) \leq \frac{B(n)}{L}.$
\end{itemize}
Once we have this crucial decomposition, we can develop the same proof as in the standard situation (choosing the number $c_{Q_j}$ according to the same criterion)  to obtain
\begin{eqnarray*}
\avgint_I |f-c_Q| & \leq & \varphi^{-1} \big(L) + \frac{1}{\mu(Q)} \sum_j \mu(Q_j)\Big(  \avgint_{Q_j} |f-c_{Q_j}| + |c_Q-c_{Q_j}|\Big)\\
 & \leq & \varphi^{-1}\big(L) + \frac{1}{\mu(Q)} \sum_j \mu(Q_j)\Big(  X + \varphi^{-1} \big( 2L+2\big)\Big) \\
 & \leq & \varphi^{-1} \big(L)  + \frac{B(n)}{L}X+\frac{B(n)}{L}\varphi^{-1} \big( 2L+2\big)
\end{eqnarray*}
Finally, taking the supremum over all cubes on the left hand side and choosing $L=2B(n)$ we obtain 
\begin{equation*}
X \leq 2\varphi^{-1} \big(2B(n))  + \varphi^{-1} \big(4B(n)+2),
\end{equation*}
which finishes the proof. The assumption that $X<\infty$ can be done using the same truncation argument as before.

\subsection{The proof for rectangles and non-doubling measures}\label{sec:rectangles}

At this point, the only important thing is to show that we do have an appropriate decomposition lemma. We include here the statement of the aforementioned Rising sun lemma.

\begin{lemma}[Riesz's Rising Sun]\label{lem:Sun}
Let $R$ be a rectangle in $\R^n$ and let $\mu$ be a measure such that $\mu(\partial P)=0$ for any rectangle $P$ (for example, a measure satisfying \eqref{eq:Measure-Growht} or, more generally, any non-atomic measure). Let $h$ be a real function in $L^1_{\mu}(R)$  and let $\lambda >h_R$. There exist an at most countable family of pairwise disjoint rectangles $R_j\subset R$ such that $h_{R_j}=\lambda$ and $h(x) \leq \lambda$ for almost every $x\in R\setminus \bigcup_j R_j.$
\end{lemma}

Moreover, the total mass of the selected cubes cannot be too big, meaning that if $h \geq 0,$
\begin{equation*}
\sum_j \mu(R_j) = \sum_j \frac 1\lambda \int_{R_j}h\, d\mu \leq \frac{\mu(R)}{\lambda} \avgint_Rh \, d\mu.
\end{equation*}

Equipped with this lemma, the rest of the proof of Theorem \ref{thm:rectangles-non-doubling} follows the exact same steps as in Theorem \ref{thm:Main-Cubes}. The relevant quantity is of course 

\begin{equation}\label{eq:X-rectangles-non-doubling-R^n}
X = \sup_{R \text{ rectangle } }  \avgint_R |f(x)-c_R|\ d\mu,
\end{equation}
where $c_R$ is a constant such that 
\[
\avgint_R \varphi(f-c_R)d\mu \leq 2.
\]
Note that Lemma \ref{lem:Sun} is particularly well adapted to the setting of rectangles (and not useful for cubes) since the decomposition is \textbf{always} into rectangles, even if we start with a cube (see the discussion in \cite{KLS}). Then, when intercalating the average of the form 
\begin{equation*}
\avgint_{R_j} |f(x)-c_{R_j}|\ d\mu.
\end{equation*}
from the decomposition, we can control it by using our $X$ defined in \eqref{eq:X-rectangles-non-doubling-R^n}, so the proof of Theorem \ref{thm:rectangles-non-doubling} can be obtained following the same line of ideas.

\bibliographystyle{amsalpha}



\end{document}